\documentclass[reqno,12pt]{amsart}
\pdfoutput=1
\usepackage{amsmath,amsthm,amsfonts,amssymb,ifpdf}
\usepackage{amssymb}
\usepackage{amsmath}
\usepackage{amsthm}
\usepackage{graphicx}
\usepackage[all]{xy}
\usepackage{enumerate}
\usepackage{tikz-cd}

\newcommand{\RomanNumeralCaps}[1]
    {\MakeUppercase{\romannumeral #1}}
\theoremstyle{plain} 
\newtheorem{theorem}{Theorem}[section]
\newtheorem{proposition}[theorem]{Proposition}
\newtheorem{lemma}[theorem]{Lemma}
\newtheorem{corollary}[theorem]{Corollary}
\newtheorem{question}[theorem]{Question}

\theoremstyle{definition} \newtheorem{definition}[theorem]{Definition}

\theoremstyle{remark} \newtheorem{remark}[theorem]{Remark}

\usepackage[all]{xy}

\ifpdf
  \usepackage[
    pdftex,
    colorlinks,%
    linkcolor=blue,citecolor=red,urlcolor=blue,
    hyperindex,%
    plainpages=false,%
    bookmarksopen,%
    bookmarksnumbered%
  ]{hyperref} 
 
 \usepackage{thumbpdf}
\else
  \usepackage{hyperref}
\fi

\def\Num{{\rm Num}}
\def\Div{{\rm Div}}

\def\Sym{{\rm Sym}}
\def\Pic{{\rm Pic}}
\def\Gr{{\rm Gr}}
\def\cont{{\rm cont}}
\def\dim{{\rm dim}}

\def\Nef{{\rm Nef}}
\def\NE{{\rm NE}}

\newcommand{\ncom}{\newcommand}
\ncom{\mylabel}[1]{{\rm (#1)}\label{#1}}
\ncom{\Hom}{{\textit{Hom}}}
\ncom{\eop}{{\hfill $\Box$}}
\begin{document}
\baselineskip=16pt

\setcounter{tocdepth}{1}

\title[Products of Fano varieties with Picard number one]{Classification of products of Fano varieties with Picard number one}

\author{Arijit Mukherjee}
\address{Department of Mathematics\\ Indian Institute of Technology Madras\\ Tamil Nadu - 600 036, India.}
\email{mukherjee7.arijit@gmail.com}

\begin{abstract}
Given a partition $(n_1,\ldots,n_r)$ of a positive integer $n$, one has the associated $n$-dimensional multiprojective space $\mathbb{P}^{n_1}\times \cdots \times \mathbb{P}^{n_r}$.  We show that distinct partitions of $n$ yield non-isomorphic multiprojective spaces, giving a new proof via the extremal contractions of their closed cone of curves.  In contrast to the earlier approaches, the argument here is uniform across all partitions, and extends beyond multiprojective spaces.  In fact, we further extend it to a more general setting, namely to products of Fano varieties of Picard number one: we prove that a fixed such factor in each dimension makes the products attached to distinct partitions pairwise non-isomorphic.  As a consequence, a complete classification of products of smooth quadrics, of dimension $\geq 3$, has been obtained.
\end{abstract}
\maketitle
\tableofcontents

\section{Introduction}
\label{sec:Introduction}
Let $d$ be any positive integer.  By $\mathbb{P}^d$ we denote the complex projective space of dimension $d$. Referring to the product of projective spaces by multiprojective spaces following the terminology of C. Musili (cf. \cite[Chapter 3, \S 36, p. 150]{Mus}), the following question arises quite naturally.  
\begin{question}\label{question}
Let $n$ be a positive integer. Given any two distinct partitions $(m_1,m_2,\ldots,m_r)$ and $(n_1,n_2,\ldots,n_s)$ of $n$, are the corresponding $n$-dimensional multiprojective spaces $\mathbb{P}^{m_1}\times\cdots\times \mathbb{P}^{m_r}$ and $\mathbb{P}^{n_1}\times \cdots\times\mathbb{P}^{n_s}$ non-isomorphic?
\end{question}
A complete solution to Question \ref{question} is given in \cite{MN}.  In that paper, the authors answered the question by breaking it into two mutually exclusive and exhaustive cases, namely for partitions of same length and partitions of different length.  The latter case has been handled easily by calculating the Picard group of the corresponding multiprojective spaces, whereas the former case has been dealt with using the description of the nef cones of the spaces involved (cf. \cite[Proposition 4.7 and Theorem 4.12]{MN}).  Another solution to the posed question has been provided in \cite[Theorem 3.5]{Mu} using a decomposition of tensor products of irreducible representations of simple Lie algebras.  

Question \ref{question} can be asked at several other generalised context.  
\begin{question}\label{generalised question}
Let $(m_1,\ldots,m_r)$ and $(n_1,\ldots,n_s)$ be any two distinct partitions of a positive integer $n$. 
\begin{enumerate}[(a)]
    \item Let $C$ be a smooth projective curve over $\mathbb{C}$ and $\Sym^d(C)$ be its $d$-th symmetric product. Are the corresponding $n$-dimensional multi symmetric products $\Sym^{m_1}(C)\times\cdots\times \Sym^{m_r}(C)$ and $\Sym^{n_1}(C)\times \cdots\times\Sym^{n_s}(C)$ non-isomorphic? 
    \item Let $\mathcal{F}_1^d$ be a fixed Fano variety with Picard number $1$ of dimension $d$.  Are the corresponding $n$-dimensional spaces $\mathcal{F}_1^{m_1}\times\cdots\times \mathcal{F}_1^{m_r}$ and $\mathcal{F}_1^{n_1}\times \cdots\times\mathcal{F}_1^{n_s}$ non-isomorphic?
\end{enumerate}
\end{question}
Clearly, Question \ref{question} is $C=\mathbb{P}^1$ version of Question \ref{generalised question}(a).  Question \ref{generalised question}(a) has been completely answered in \cite[Proposition 4.11]{MN}.  For even further progress in this direction, one can look at \cite{DMP} and \cite{Sarkar}. 

Note that to make sense of Question \ref{generalised question}(b), one needs to take $\mathcal{F}_1^d$ as a ``fixed'' Fano variety of dimension $d$ and Picard number $1$ as there are many such non-isomorphic Fano varieties of same dimension.  For example, a list of non-isomorphic Fano threefolds of Picard number one can be found in \cite[\S 12.2, pp.~214-215]{IvPr}.  Therefore, unless specified, such products of Fano varieties can be non-isomorphic even if the partitions are the same.  This is not at all the case for Question \ref{generalised question}(a), hence for Question \ref{question}, as given a curve $C$ and a positive integer $d$ there is only one $\Sym^d(C)$, up to isomorphism.

In this paper, we provide another alternative solution to the Question \ref{question} (cf. Theorem \ref{Main Theorem_1_with proof}).  We prove it by calculating the extremal contractions of the cone of curves of the multiprojective spaces involved.  This technique is very useful from the context of Minimal Model Program.  One can look at \cite{KM} and \cite{M} for more details about this machinery and its applications.  One notable advantage of this method over the methods used in \cite{MN} and \cite{Mu} is that it is uniform across all partitions, and extends beyond multiprojective spaces.  In fact, using the technique mentioned along with Kodaira vanishing theorem and Cone theorem, we further answer Question \ref{generalised question}(b) (cf. Theorem \ref{Main Theorem_2_with proof}).  This is immediately useful in classifying several important spaces like products of Grassmannians (cf. Remark \ref{Remark about Grassmannians and Fixed det moduli spaces}), products of quadrics of dimension $\geq 3$ (cf. Corollary \ref{quadric corollary}), products of fixed determinant moduli spaces of stable bundles over any smooth projective curve of genus $g\geq 2$ (cf. Remark \ref{Remark about Grassmannians and Fixed det moduli spaces}), etc.  

\section{Preliminaries}
In this section, we briefly recall some notions and results along with notations which are useful to prove our main theorems.

Let $X$ be a complete variety over complex numbers.  We denote the group of all (Cartier) divisor on $X$ by $\Div(X)$.  Two divisors $D_1$ and $D_2$ are said to be numerically equivalent, if for all irreducible curve $C\subseteq X$,
\begin{equation*}
(D_1\cdot C)=(D_2\cdot C).
\end{equation*}
Moreover, a divisor is said to be numerically trivial if it is numerically equivalent to zero.  The subgroup of $\Div(X)$ of all numerically trivial divisors is denoted by $\Num(X)$.  The corresponding quotient $\tfrac{\Div(X)}{\Num(X)}$, denoted by $N^1(X)$, is called the N\'eron-Severi group of $X$.  It is well known that $N^1(X)$ is a free abelian group of finite rank (cf. \cite[Proposition 1.1.16, p.~18]{L}).  By Picard number of $X$, denoted by $\rho(X)$, we mean the rank of $N^1(X)$. 

We now recall what one means by the nef cone of a variety and its dual cone.
\begin{definition}
Let $X$ be a complete variety.  Let $L$ be a line bundle on $X$ and $c_1(L)\in H^2(X,\mathbb{Z})$ be its cohomology class.  Then $L$ is said to be numerically effective, or nef in short, if 
\begin{equation*}
\int_Cc_1(L)\geq 0,
\end{equation*}
for every irreducible curve $C$ of $X$.  Equivalently, A (Cartier) divisor $D$ on $X$ is said to be nef if 
\begin{equation*}
(D\cdot C)\geq 0,
\end{equation*}
for every irreducible curve $C$ of $X$. 
\end{definition}
\begin{definition} 
A subset $K$ of a finite dimensional real vector space $V$ is said to be a cone if it is stable under multiplication by positive scalars.
\end{definition}
\begin{definition}
Let $X$ be a complete variety.  The cone $\Nef(X)$ of all nef $\mathbb{R}$-divisor classes $N^1(X)_{\mathbb{R}}(:=N^1(X)\otimes \mathbb{R})$ is called the nef cone of $X$. 
\end{definition}
We now recall the meaning of cone of curves of $X$, as a dual cone of $\Nef(X)$.  Let ${Z_1(X)}_{\mathbb{R}}$ be the real vector space of real one cycles on $X$.  Two such one cycles $\gamma_1$ and $\gamma_2$ are said to be numerically equivalent if for all $D\in \Div_{\mathbb{R}}(X)(:=\Div(X)\otimes \mathbb{R})$,
\begin{equation*}
\bigl(D\cdot\gamma_1\bigr)=\bigl(D\cdot\gamma_2\bigr).
\end{equation*}
Then the corresponding quotient vector space of numerically equivalent one cycles is denoted by ${N_1(X)}_{\mathbb{R}}$.
\begin{definition} 
Let $X$ be a complete variety.  The cone of curves $\NE(X)(\subseteq {N_1(X)}_{\mathbb{R}})$ of $X$ is the cone spanned by the classes of effective one cycles on $X$, i.e,
\begin{equation*}
\NE(X):=\Bigl\{\sum a_i[C_i]\mid C_i\subset X\;\text{irreducible\; curve\;}, a_i\geq 0\Bigr\}.
\end{equation*}
The closure (w.r.t standard Euclidean topology) $\overline{\NE}(X)$ of $\NE(X)$ is called the closed cone of curves on $X$.
\end{definition}

That brings us to the most important machinery, namely contraction of extremal faces of the closed cone curves of $X$, we use to prove the main results of this paper.  Precise meaning of the same is as follows.
\begin{definition}\label{extremal contraction of a face}
Let $X$ be a projective variety and $F$ be an extremal face of $\overline{\NE}(X)$.  Then a morphism $\cont_F:X\rightarrow Z$, $Z$ being a projective variety, is said to be an extremal contraction of $F$ if the following conditions are satisfied:
\begin{enumerate}
\item For an irreducible curve $C$ in $X$, $\cont_F(C)=\{\cdot\}$ if and only if $[C]\in F$.
\item $(\cont_F)_{\ast}\mathcal{O}_X=\mathcal{O}_Z$.
\end{enumerate}
\end{definition}
\begin{remark}\label{uniquesness of extremal contraction}
\begin{enumerate}
\item An extremal contraction $\cont_F$, as in Definition \ref{extremal contraction of a face}, are uniquely determined by the face $F$ (cf. \cite[Remark 1.26, p.~25]{KM}).
\item Not every face can have an extremal contraction (cf. \cite[Example 1.27, p.~25-26]{KM}).
\end{enumerate}  
\end{remark}
As mentioned in Remark \ref{uniquesness of extremal contraction}, extremal contraction may not always exist.  We now recall a condition which forces the existence of the same.
\begin{theorem}\label{existence of extremal contractions}
Let $X$ be a projective variety and $L$ a nef line bundle on $X$.  Let $F_L$ be a face of $\overline{\NE}(X)$ associated to $L$, defined as
\begin{equation*}
F_L:= \overline{\NE}(X)\cap L^{\perp}=\Bigl\{[C]\in N_1(X)_{\mathbb{R}}\mid (C\cdot L)=0\Bigr\}.
\end{equation*}
If 
\begin{equation*}
F_L\subset \overline{\NE}(X)_{K_X<0}=\Bigl\{[C]\in N_1(X)_{\mathbb{R}}\mid (C\cdot K_X)<0\Bigr\},
\end{equation*}
$K_X$ being the canonical line bundle of $X$, then there exists an extremal contraction $\cont_{F_L}: X\rightarrow Y$.
\end{theorem}
\begin{proof}
See \cite[Theorem 8.1.3 and Exercise 8.1.4, p.~302-304]{M}.  For a more general version see \cite[Theorem 3.7, p.~76]{KM}.
\end{proof}

\section{Proof of Main Theorems}
In this section, we classify the multiprojective spaces and product of Fano varieties with Picard number one.  Towards that, we first calculate the nef cone and closure of its dual cone of any multiprojective spaces.
\begin{lemma}\label{Nef cone of multiprojective spaces}
Let $n_1,n_2,\cdots,n_r$ be $r$ many positive integers and\begin{equation*}
pr_{n_i}: \mathbb{P}^{n_1}\times 
\mathbb{P}^{n_2}\times \cdots \times \mathbb{P}^{n_r} \longrightarrow \mathbb{P}^{n_i}
\end{equation*}
be the $i$-th projection map, for all $1\leq i\leq r$. Then
\begin{equation*}
\Nef(\mathbb{P}^{n_1}\times\cdots \times \mathbb{P}^{n_r})=\mathbb{R}_{\geq 0}[pr_{n_1}^{\ast}(\mathcal{O}_{\mathbb{P}^{n_1}}(1))]+\cdots+\mathbb{R}_{\geq 0}[pr_{n_r}^{\ast}(\mathcal{O}_{\mathbb{P}^{n_r}}(1))].
\end{equation*}
\end{lemma}
\begin{proof}
For any $L\in \Pic(\mathbb{P}^{n_1}\times\cdots \times \mathbb{P}^{n_r})$, by $[L]$ we denote its image in $ N^1(\mathbb{P}^{n_1}\times\cdots \times \mathbb{P}^{n_r})_{\mathbb{R}}$.  So, there exists $a_i\in \mathbb{R}, 1\leq i \leq r$, such that 
\begin{equation*}
[L]=a_1[pr_{n_1}^{\ast}(\mathcal{O}_{\mathbb{P}^{n_1}}(1))]+\cdots+a_r[pr_{n_r}^{\ast}(\mathcal{O}_{\mathbb{P}^{n_r}}(1))].
\end{equation*}
Moreover, $[L]\in \Nef(\mathbb{P}^{n_1}\times\cdots \times \mathbb{P}^{n_r})$ if and only if $(L\cdot C)\geq 0$ for all irreducible curves $C\subset \mathbb{P}^{n_1}\times\cdots \times \mathbb{P}^{n_r}$.  In particular, we choose $C_1=\mathbb{P}^1\times\{\cdot\}\times \cdots \times \{\cdot\}$.  Then,
\begin{equation*}
\bigl(L\cdot C_1\bigr)=\bigl(a_1[pr_{n_1}^{\ast}(\mathcal{O}_{\mathbb{P}^{n_1}}(1))]+\cdots+a_r[pr_{n_r}^{\ast}(\mathcal{O}_{\mathbb{P}^{n_r}}(1))]\cdot[\mathbb{P}^1\times\{\cdot\}\times \cdots \times \{\cdot\}]\bigr)=a_1.   
\end{equation*}
Similarly, for all $1\leq i \leq r$, we have $(L\cdot C_i)=a_i$.  Therefore, $[L]\in \Nef(\mathbb{P}^{n_1}\times\cdots \times \mathbb{P}^{n_r})$ if and only if $a_i\geq 0$ for all $1\leq i\leq r$.  Hence the assertion follows.
\end{proof}
\begin{corollary}\label{Mori cone of multiprojective space}
Let $n_1,n_2,\cdots,n_r$ be $r$ many positive integers. Then
\begin{equation*}
\overline{\NE}(\mathbb{P}^{n_1}\times\cdots \times \mathbb{P}^{n_r})=\mathbb{R}_{\geq 0}[\mathbb{P}^1\times\{\cdot\}\times \cdots \times \{\cdot\}]+\cdots+\mathbb{R}_{\geq 0}[\{\cdot\}\times \cdots \times \{\cdot\}\times \mathbb{P}^1].
\end{equation*}
\end{corollary}
\begin{proof}
Follows directly from Lemma \ref{Nef cone of multiprojective spaces}.
\end{proof}
\begin{remark}
Identifying $\Pic(\mathbb{P}^{n_1}\times\cdots \times \mathbb{P}^{n_r})$ with direct sum of $r$ copies of $\mathbb{Z}$, the line bundle $pr_{n_i}^{\ast}(\mathcal{O}_{\mathbb{P}^{n_i}}(1))$ is nothing but $(0,0,\ldots,1,\dots,0)$, $1$ being in the $i$-th place.
\end{remark}
We now classify multiprojective spaces after calculating the extremal contractions of some faces of the closed cone of curves of the corresponding spaces.
\begin{proposition}\label{extremal contractions of multiprojective spaces}
Let $F_{(0,0,\ldots,1,\dots,0)}$ be a face of $\overline{\NE}(\mathbb{P}^{n_1}\times\cdots \times \mathbb{P}^{n_r})$ associated to the (nef) line bundle $(0,0,\ldots,1,\dots,0)$ on $\mathbb{P}^{n_1}\times\cdots \times \mathbb{P}^{n_r}$, $1$ being in the $i$-th place.  Then the projections 
\begin{equation*}
pr_{n_i}: \mathbb{P}^{n_1}\times 
\mathbb{P}^{n_2}\times \cdots \times \mathbb{P}^{n_r} \longrightarrow \mathbb{P}^{n_i}
\end{equation*}
are the extremal contractions of the faces $F_{(0,0,\ldots,1,\dots,0)}$, for all $1\leq i \leq r$. 
\end{proposition}
\begin{proof}
Let us denote the line bundle $(0,0,\ldots,1,\dots,0)$ and the curve $\{\cdot\}\times \cdots \times \mathbb{P}^1 \times \cdots \times \{\cdot\}$, $1$ being in the $i$-th place, by $L_i$ and $C_i$ respectively.  Then,
\begin{equation}\label{description of the faces}
\begin{split}
F_{L_i}&=\overline{\NE}(\mathbb{P}^{n_1}\times\cdots \times \mathbb{P}^{n_r})\cap L_i^{\perp}\\
&=\Bigl\{[C]\in N_1(\mathbb{P}^{n_1}\times\cdots \times \mathbb{P}^{n_r})_{\mathbb{R}}\mid (C\cdot L_i)=0\Bigr\}\\
&=+_{j\neq i}\mathbb{R}_{\geq 0}[C_j].
\end{split} 
\end{equation}
Therefore, for any $[C] \in F_{L_i}$,
\begin{equation*}
\begin{split}
\bigl(C\cdot K_{\mathbb{P}^{n_1}\times\cdots \times \mathbb{P}^{n_r}}\bigr)&=\bigl(C\cdot (-n_1-1,\ldots,-n_r-1)\bigr)\\
&=-\bigl(C\cdot (n_1+1,\ldots,n_r+1)\bigr)<0.
\end{split}
\end{equation*} 
Therefore,
\begin{equation*}
F_{L_i}\subseteq \overline{\NE}(\mathbb{P}^{n_1}\times\cdots \times \mathbb{P}^{n_r})_{K_{\mathbb{P}^{n_1}\times\cdots \times \mathbb{P}^{n_r}}<0}.
\end{equation*}
Hence, by Theorem \ref{existence of extremal contractions}, extremal contractions of the faces $F_{L_i}$ exist for all $i$. 

We now have the following claim.\\
\textit{Claim} : $\cont_{F_{L_i}}=pr_{n_i}.$\\
Firstly, by \eqref{description of the faces}, it is clear that for all $1\leq i \leq r$, $pr_{n_i}(C)=\{\cdot\}$ if and only if $[C]\in F_{L_i}$, for any irreducible curve $C$ in $\mathbb{P}^{n_1}\times\cdots \times \mathbb{P}^{n_r}$.  Secondly, as $pr_{n_i}: \mathbb{P}^{n_1}\times\cdots \times \mathbb{P}^{n_r} \rightarrow \mathbb{P}^{n_i}$'s are birational projective morphisms of noetherian integral schemes and $\mathbb{P}^{n_i}$'s are normal, by \cite[proof of Corollary 11.4 (Zariski's Main Theorem), p.~280]{Ha} we have
\begin{equation*}
{(pr_{n_i})}_{\ast} \mathcal{O}_{\mathbb{P}^{n_1}\times\cdots \times \mathbb{P}^{n_r}}=\mathcal{O}_{\mathbb{P}^{n_i}},
\end{equation*}
for all $1\leq i \leq r$.  Hence the claim follows.
\end{proof} 
\begin{theorem}\label{Main Theorem_1_with proof}
Let $n$ be any positive integer.  Given any two distinct partitions $(m_1,m_2,\ldots,\\m_r)$ and $(n_1,n_2,\ldots,n_s)$ of $n$, corresponding multiprojective spaces $\mathbb{P}^{m_1}\times\cdots\times \mathbb{P}^{m_r}$ and $\mathbb{P}^{n_1}\times \cdots\times\mathbb{P}^{n_s}$ are non-isomorphic.
\end{theorem}
\begin{proof}
By Proposition \ref{extremal contractions of multiprojective spaces}, the set of extremal contractions of faces $F_{L_i}$'s of $\overline{\NE}(\mathbb{P}^{m_1}\times\cdots\times \mathbb{P}^{m_r})$ is $\{pr_{m_i}\mid 1\leq i \leq r\}$, whereas the same for $\overline{\NE}(\mathbb{P}^{n_1}\times\cdots\times \mathbb{P}^{n_s})$ is $\{pr_{n_j}\mid 1\leq j \leq s\}$.  Therefore, if the partitions $(m_1,m_2,\ldots,m_r)$ and $(n_1,n_2,\ldots,n_s)$ of $n$ are distinct, then the corresponding set of extremal contractions are different.  Hence, by Remark \ref{uniquesness of extremal contraction}, the corresponding faces of $\overline{\NE}(\mathbb{P}^{m_1}\times\cdots\times \mathbb{P}^{m_r})$ and $\overline{\NE}(\mathbb{P}^{n_1}\times\cdots\times \mathbb{P}^{n_s})$ are different and so are the cone of curves themselves.  Therefore, the multiprojective spaces $\mathbb{P}^{m_1}\times\cdots\times \mathbb{P}^{m_r}$ and $\mathbb{P}^{n_1}\times \cdots\times\mathbb{P}^{n_s}$ are non-isomorphic.  
\end{proof}
We want to replicate the same classification result in more general set up, namely for Fano varieties with Picard number one.  We recall the precise definition of a Fano variety (cf. \cite[Section 0, p.~681]{Nad}).
\begin{definition}\label{Fano variety_definition}
A smooth connected projective variety $X$ over $\mathbb{C}$ is called a Fano variety if its anticanonical bundle $-K_X$ is ample.
\end{definition}
\begin{remark}
\begin{enumerate}
\item Though many mathematicians have considered Fano varieties with certain singularities, namely terminal and klt singularities, while studying Minimal Model Program, throughout this paper we assume a Fano variety to be smooth as mentioned in Definition \ref{Fano variety_definition}.
\item Fano varieties are not only connected, they are in fact rationally connected, i.e, any two general points can be joined by an (irreducible) rational curve, (cf. \cite{KMM}).  
\end{enumerate}
\end{remark}  
To prove the classification Theorem \ref{Main Theorem_1_with proof} for Fano varieties with Picard number one following the same procedure, we wish for Lemma \ref{Nef cone of multiprojective spaces}, Corollary \ref{Mori cone of multiprojective space} and Proposition \ref{extremal contractions of multiprojective spaces} to hold in this set up as well. One of the main results used to calculate the Nef cone of the multiprojective space $\mathbb{P}^{n_1}\times \cdots\times\mathbb{P}^{n_s}$ is that        
$\Pic(\mathbb{P}^{n_1}\times\cdots \times \mathbb{P}^{n_r})$ is the direct sum of $\Pic(\mathbb{P}^{n_i})$s, $1\leq i \leq r$.  This splitting of Picard group of multiprojective spaces in terms of product of Picard group of individual factors holds as the factor $\mathbb{P}^i$s are rational, (cf. \cite[Chapter \RomanNumeralCaps{2}, Corollary 6.16, p.~145 \& Exercise 6.1, p.~146]{Ha}).  As Fano varieties are not rational in general (cf. \cite{IvPr}), a priori it seems that a similar result about splitting of Picard group may not hold in this case.  We show that is not the case.  In that regard, we first recall the following result.
\begin{proposition}\label{Pic of prod is prod of Pics}
Let $X$ be an integral projective scheme over an algebraically closed field $k$ with $H^1(X,\mathcal{O}_X)=0$.  Let $Y$ be a connected scheme of finite type over $k$.  Then 
\begin{equation*}
\Pic(X\times Y)=\Pic(X)\times \Pic(Y).
\end{equation*}
\end{proposition}
\begin{proof}
See \cite[Chapter \RomanNumeralCaps{3}, Exercise 12.6, p.~292]{Ha}.
\end{proof}
We now check one of the conditions mentioned in the hypothesis of Proposition \ref{Pic of prod is prod of Pics} for Fano varieties.
\begin{lemma}\label{vanishing of cohomologies for Fano variety}
Let $X$ be a Fano variety over $\mathbb{C}$.  Then for any positive integer $j$,
\begin{equation*}
H^j(X,\mathcal{O}_X)=0. 
\end{equation*} 
\end{lemma}
\begin{proof}
 Applying the Kodaira vanishing theorem (cf. \cite[Theorem 4.2.1, p.~248-249]{L}) to the ample line bundle $-K_X$, we obtain
\begin{equation*}
H^j(X,\mathcal{O}_X)=H^j(X,K_X\otimes(-K_X))=0\;\text{for\;all\;}j>0.
\end{equation*}
\end{proof} 
\begin{proposition}\label{Pic of prod is prod of Pics for Fano varieties}
Let $X$ and $Y$ be two Fano varieties over $\mathbb{C}$.  Then,
\begin{equation*}
\Pic(X\times Y)=\Pic(X)\times \Pic(Y).
\end{equation*}
Moreover, if $X$ and $Y$ are of Picard number one, then
\begin{equation*}
\Pic(X\times Y)=\mathbb{Z}\oplus \mathbb{Z}.
\end{equation*}   
\end{proposition}
\begin{proof}
First part follows from Proposition \ref{Pic of prod is prod of Pics} and Lemma \ref{vanishing of cohomologies for Fano variety}.  Second part is then obvious. 
\end{proof}
We now record the description of the closed cone of curves of a Fano variety with Picard number one, which is the analogue of Corollary \ref{Mori cone of multiprojective space} needed in this general set up.
\begin{lemma}\label{Mori cone of Fano with Picard number one}
Let $X$ be a Fano variety with $\rho(X)=1$.  Then $N_1(X)_{\mathbb{R}}\cong \mathbb{R}$ and there exists a rational curve $\ell\subset X$ such that
\begin{equation*}
\overline{\NE}(X)=\mathbb{R}_{\geq 0}[\ell].
\end{equation*}
\end{lemma}
\begin{proof}
As $X$ is Fano, $-K_X$ is ample and hence $K_X\cdot z<0$ for all nonzero $z\in \overline{\NE}(X)$.  That is,   $\overline{\NE}(X)_{K_X\geq 0}$ becomes trivial.  Therefore, by the Cone Theorem (cf. \cite[Theorem 1.5.33, p.~86]{L}), closed cone of curves $\overline{\NE}(X)\subseteq N_1(X)_{\mathbb{R}}$ is a finite rational polytope spanned by the classes of rational curves (cf. \cite[Example 1.5.34, p.~87]{L}).  Moreover, as $\dim_{\mathbb{R}}N_1(X)_{\mathbb{R}}=\dim_{\mathbb{R}}N^1(X)_{\mathbb{R}}=\rho(X)=1$, there exists some rational curve $\ell\subset X$ such that $\overline{\NE}(X)=\mathbb{R}_{\geq 0}[\ell]$.
\end{proof}
Let $\mathcal{F}_1^{d_1},\ldots,\mathcal{F}_1^{d_r}$ be Fano varieties with Picard number one, of dimensions $d_1,\ldots,d_r$ respectively, and let
\begin{equation*}
pr_{d_i}: \mathcal{F}_1^{d_1}\times \cdots \times \mathcal{F}_1^{d_r}\longrightarrow \mathcal{F}_1^{d_i}
\end{equation*}
denote the $i$-th projection, for all $1\leq i \leq r$.  Let $H_i$ denotes an ample generator of $\Pic(\mathcal{F}_1^{d_i})\cong \mathbb{Z}$, $1\leq i \leq r$.  By Proposition \ref{Pic of prod is prod of Pics for Fano varieties}, we have $\Pic(\mathcal{F}_1^{d_1}\times \cdots \times \mathcal{F}_1^{d_r})\cong \mathbb{Z}^{\oplus r}$, generated by the pullbacks $pr_{d_i}^{\ast}(H_i)$.  By Lemma \ref{Mori cone of Fano with Picard number one}, for each $i$ let $\ell_i\subset \mathcal{F}_1^{d_i}$ be a rational curve spanning $\overline{\NE}(\mathcal{F}_1^{d_i})$, and let $\widetilde{\ell}_i:=\{\cdot\}\times \cdots \times \ell_i \times \cdots \times \{\cdot\}$ be the corresponding curve in the product, $\ell_i$ being in the $i$-th place.  With these notations, the analogues of Lemma \ref{Nef cone of multiprojective spaces}, Corollary \ref{Mori cone of multiprojective space} and Proposition \ref{extremal contractions of multiprojective spaces} read as follows.
\begin{proposition}\label{Nef and Mori cone of product of Fano}
With the notations as above,
\begin{equation*}
\Nef(\mathcal{F}_1^{d_1}\times \cdots \times \mathcal{F}_1^{d_r})=\mathbb{R}_{\geq 0}[pr_{d_1}^{\ast}(H_1)]+\cdots +\mathbb{R}_{\geq 0}[pr_{d_r}^{\ast}(H_r)],
\end{equation*}
and
\begin{equation*}
\overline{\NE}(\mathcal{F}_1^{d_1}\times \cdots \times \mathcal{F}_1^{d_r})=\mathbb{R}_{\geq 0}[\widetilde{\ell}_1]+\cdots +\mathbb{R}_{\geq 0}[\widetilde{\ell}_r].
\end{equation*}
Moreover, the projections $pr_{d_i}$ are the extremal contractions of the faces $F_{pr_{d_i}^{\ast}(H_i)}$ associated to the nef line bundles $pr_{d_i}^{\ast}(H_i)$, for all $1\leq i \leq r$.
\end{proposition}
\begin{proof}
By Proposition \ref{Pic of prod is prod of Pics for Fano varieties}, the class of any line bundle $L$ on the product can be written as $[L]=a_1[pr_{d_1}^{\ast}(H_1)]+\cdots+a_r[pr_{d_r}^{\ast}(H_r)]$ with $a_i\in \mathbb{R}$.  By the projection formula, $\bigl(L\cdot \widetilde{\ell}_i\bigr)=a_i\,(H_i\cdot \ell_i)$ for all $i$, and $(H_i\cdot \ell_i)>0$ since $H_i$ is ample.  Hence $[L]$ is nef if and only if $a_i\geq 0$ for all $i$, which gives the description of $\Nef(\mathcal{F}_1^{d_1}\times \cdots \times \mathcal{F}_1^{d_r})$.  The description of $\overline{\NE}(\mathcal{F}_1^{d_1}\times \cdots \times \mathcal{F}_1^{d_r})$ follows by duality, exactly as in Lemma \ref{Nef cone of multiprojective spaces} and Corollary \ref{Mori cone of multiprojective space}, using Lemma \ref{Mori cone of Fano with Picard number one} for each factor.

For the last assertion, writing $L_i:=pr_{d_i}^{\ast}(H_i)$ and $F_{L_i}=\overline{\NE}(\mathcal{F}_1^{d_1}\times \cdots \times \mathcal{F}_1^{d_r})\cap L_i^{\perp}=+_{j\neq i}\mathbb{R}_{\geq 0}[\widetilde{\ell}_j]$, the anticanonical bundle $-K_{\mathcal{F}_1^{d_1}\times \cdots \times \mathcal{F}_1^{d_r}}=\bigotimes_{i=1}^rpr_{d_i}^{\ast}(-K_{\mathcal{F}_1^{d_i}})$ is ample, being a tensor product of pullbacks of the ample bundles $-K_{\mathcal{F}_1^{d_i}}$.  Hence $(C\cdot K_{\mathcal{F}_1^{d_1}\times \cdots \times \mathcal{F}_1^{d_r}})<0$ for every $0\neq [C]\in \overline{\NE}(\mathcal{F}_1^{d_1}\times \cdots \times \mathcal{F}_1^{d_r})$ and in particular $F_{L_i}\subseteq \overline{\NE}(\mathcal{F}_1^{d_1}\times \cdots \times \mathcal{F}_1^{d_r})_{K_{\mathcal{F}_1^{d_1}\times \cdots \times \mathcal{F}_1^{d_r}}<0}$.  By Theorem \ref{existence of extremal contractions}, the extremal contraction of $F_{L_i}$ exists and coincides with $pr_{d_i}$ exactly as in the proof of Proposition \ref{extremal contractions of multiprojective spaces}.
\end{proof}
Finally, we have the following classification result for product of Fano varieties with Picard number one.    
\begin{theorem}\label{Main Theorem_2_with proof}
Let $n$ be any positive integer.  Given any two distinct partitions $(m_1,m_2,\ldots,\\m_r)$ and $(n_1,n_2,\ldots,n_s)$ of $n$, corresponding product spaces $\mathcal{F}_1^{m_1}\times\cdots\times \mathcal{F}_1^{m_r}$ and $\mathcal{F}_1^{n_1}\times \cdots\times\mathcal{F}_1^{n_s}$ are non-isomorphic.
\end{theorem}
\begin{proof}
By Proposition \ref{Nef and Mori cone of product of Fano}, the set of extremal contractions of the faces $F_{L_i}$'s of $\overline{\NE}(\mathcal{F}_1^{m_1}\times\cdots\times \mathcal{F}_1^{m_r})$ is $\{pr_{m_i}\mid 1\leq i \leq r\}$, whereas the same for $\overline{\NE}(\mathcal{F}_1^{n_1}\times\cdots\times \mathcal{F}_1^{n_s})$ is $\{pr_{n_j}\mid 1\leq j \leq s\}$.  Therefore, if the partitions $(m_1,m_2,\ldots,m_r)$ and $(n_1,n_2,\ldots,n_s)$ of $n$ are distinct, then the corresponding set of extremal contractions are different.  Hence, by Remark \ref{uniquesness of extremal contraction}, the corresponding faces of $\overline{\NE}(\mathcal{F}_1^{m_1}\times\cdots\times \mathcal{F}_1^{m_r})$ and $\overline{\NE}(\mathcal{F}_1^{n_1}\times\cdots\times \mathcal{F}_1^{n_s})$ are different and so are the cone of curves themselves.  Therefore, the product spaces $\mathcal{F}_1^{m_1}\times\cdots\times \mathcal{F}_1^{m_r}$ and $\mathcal{F}_1^{n_1}\times \cdots\times\mathcal{F}_1^{n_s}$ are non-isomorphic.
\end{proof}
\begin{corollary}\label{quadric corollary}
Let $Q^d\subset \mathbb{P}^{d+1}$ denote the smooth quadric hypersurface of
dimension $d$. For any positive integer $n$ and any two
distinct partitions $(m_1,\ldots,m_r)$ and $(n_1,\ldots,n_s)$ of $n$ with
$m_i,n_j\geq 3$ for all $i,j$, the corresponding products
$Q^{m_1}\times\cdots\times Q^{m_r}$ and $Q^{n_1}\times\cdots\times Q^{n_s}$
are non-isomorphic.
\end{corollary}

\begin{proof}
Fix $d\geq 3$ and let $Q=Q^d\subset \mathbb{P}^{d+1}$ be a smooth quadric
hypersurface, that is, the zero locus of a nondegenerate quadratic form in $d+2$ variables.  It is a smooth connected projective variety of dimension
$d$.

By the adjunction formula (cf. \cite[Chapter 1, p.~147]{GH}), $K_Q=\mathcal{O}_{\mathbb{P}^{d+1}}(-d-2+2)\big|_Q=\mathcal{O}_Q(-d)$.  Hence, $-K_Q=\mathcal{O}_Q(d)$,
where $\mathcal{O}_Q(1)$ is the restriction of the hyperplane bundle
$\mathcal{O}_{\mathbb{P}^{d+1}}(1)$.  As $\mathcal{O}_Q(1)$ is very ample, being the restriction of the hyperplane bundle under the closed embedding
$Q\hookrightarrow \mathbb{P}^{d+1}$, its positive multiple
$-K_Q=\mathcal{O}_Q(d)$ is ample.  Thus $Q$ is a Fano variety.

Since $\dim\; Q=d\geq 3$, the Lefschetz hyperplane theorem (cf. \cite[Example 3.1.25, p.~195]{L}) yields that the restriction map
\begin{equation*}
\Pic(\mathbb{P}^{d+1})\longrightarrow \Pic(Q)
\end{equation*}
is an isomorphism.
As $\Pic(\mathbb{P}^{d+1})=\mathbb{Z}\cdot\mathcal{O}_{\mathbb{P}^{d+1}}(1)$,
we obtain $\Pic(Q)=\mathbb{Z}\cdot\mathcal{O}_Q(1)\cong \mathbb{Z}$, so that
the Picard number of $Q$ is one.

Over $\mathbb{C}$ any two nondegenerate
quadratic forms in the same number of variables are equivalent, since a
nondegenerate form can be diagonalised to $x_0^2+\cdots+x_{d+1}^2$.  Hence
any two smooth quadric hypersurfaces of dimension $d$ are projectively
equivalent, and in particular isomorphic.  Therefore, for each $d$ there is,
up to isomorphism, a unique smooth quadric $Q^d$.

For $d\geq 3$ the quadric $Q^d$ is a Fano variety with
Picard number one, so it is an admissible factor $\mathcal{F}_1^d$ in the
sense of Theorem \ref{Main Theorem_2_with proof}.  Applying that theorem to
the factors $\mathcal{F}_1^{m_i}=Q^{m_i}$ and $\mathcal{F}_1^{n_j}=Q^{n_j}$,
distinct partitions of $n$ (with all parts at least $3$) yield non-isomorphic
products $Q^{m_1}\times\cdots\times Q^{m_r}$ and
$Q^{n_1}\times\cdots\times Q^{n_s}$.  
\end{proof}

\begin{remark}
As mentioned in the proof of Corollary \ref{quadric corollary}, quadrics of a given dimension are unique up to isomorphism.  Therefore, no fixing
of factors is needed in this case.    
\end{remark}

\begin{remark}\label{Remark about Grassmannians and Fixed det moduli spaces}
\begin{enumerate}
    \item Let $\Gr(k,n)$
denote the Grassmannian parametrizing $k$-dimensional linear subspaces
$V\subset \mathbb{C}^n$, for integers $0<k<n$. It is a smooth, connected, projective variety of dimension $k(n-k)$.  Its Picard Group $\Pic(\Gr(k,n))$, being generated by the Pl\"ucker line bundle $\mathcal{O}_{\Gr(k,n)}(1)$, is isomorphic to $\mathbb{Z}$.  Moreover, its anticanonical bundle $-K_{\Gr(k,n)}$ is $\mathcal{O}_{\Gr(k,n)}(n)$.  That is, $\Gr(k,n)$ is a Fano variety with Picard number one.  More generally, any rational homogeneous projective variety $G/P$ with $G$ simple and $P$ maximal parabolic is a Fano variety with Picard number one (cf. \cite[2nd para after Remark 10.10, p.~55]{Ot} and \cite[\S 14, p.~507-513 and Corollary 14.8, p.~512]{BH}).  Therefore, products of such varieties can also be classified using Theorem \ref{Main Theorem_2_with proof}.
\item Let $C$ be a smooth projective curve over $\mathbb{C}$ of genus $g\geq 2$.  Let $r>0$ and $d$ be co-prime integers.  Then the moduli space $\mathcal{U}_C(r,L)$ of stable bundles on $C$ of rank $r$ and determinant line bundle $L$ of degree $d$ is a  Fano variety with Picard number one (cf. \cite[Theorem B(b), p.~55 and Theorem F, p.~58]{DN}).  Therefore, products of such fixed determinant moduli spaces can also be classified using Theorem \ref{Main Theorem_2_with proof}.
\end{enumerate}    
\end{remark}

\section*{Acknowledgements}
The author would like to express his gratitude to Prof. Arijit Dey, Dr. Chandranandan Gangopadhyay, Prof. J\'anos Koll\'ar and Supravat Sarkar for many useful suggestions.  The author would like to acknowledge Indian Institute of Technology Madras for financial support (Office order No.F.ARU/R10/IPDF/2024).

%\bibliographystyle{plain}
%\bibliography{bibliography.bib}

\end{document}